\documentclass[a4paper]{amsart}
\usepackage{amssymb}
\usepackage{xypic}
\usepackage{graphicx}

\newtheorem{theorem}{Theorem}[section]
\newtheorem{lemma}[theorem]{Lemma}

\newtheorem{cor}[theorem]{Corollary}

\theoremstyle{definition}
\newtheorem{definition}[theorem]{Definition}

\theoremstyle{remark}

\numberwithin{equation}{section}
\newcommand{\Q}{\mathbb{Q}}
\newcommand{\C}{\mathbb{C}}

\newcommand{\R}{\mathbb{R}}

\DeclareMathOperator{\rank}{rank}

\DeclareMathOperator{\ind}{ind}
\DeclareMathOperator{\coker}{coker}
\DeclareMathOperator{\ch}{ch}

\title{A note on torus actions and the Witten genus}

\author{Michael Wiemeler}
\address{Institut f\"ur Mathematik\\ Universit\"at Augsburg\\D-86135 Augsburg\\Germany}
\email{michael.wiemeler@math.uni-augsburg.de}
\thanks{Part of the research for this article was supported by DFG Grant HA 3160/6-1.}



\begin{document}
\begin{abstract}
  We show that the Witten genus of a string manifold \(M\) vanishes, if there is an effective action of a torus \(T\) on \(M\) such that \(\dim T>b_2(M)\).
  We apply this result to study group actions on \(M\times G/T\), where \(G\) is a compact connected Lie group and \(T\) a maximal torus of \(G\).
  
  Moreover, we use the methods which are needed to prove these results to the study of torus manifolds.
  We show that up to diffeomorphism there are only finitely many quasitoric manifolds \(M\) with the same cohomology ring as \(\#_{i=1}^k \pm\C P^n\) with \(k<n\).
\end{abstract}

\maketitle


\section{Introduction}

In this note we prove a vanishing result for the Witten genus of a string manifold on which a high dimensional torus acts effectively.
Concerning the Witten genus of string manifolds on which a compact connected Lie group acts the following is known:
\begin{itemize}
\item It has been shown by Liu \cite[discussion after Theorem 4, p.370]{MR1331972} that the Witten genus of a string manifold \(M\)  with \(b_2(M)=0\) vanishes if there is a non-trivial action of \(S^1\) on \(M\).
\item Dessai \cite{MR1731460}  showed that the Witten genus of a string manifold \(M\) vanishes if there is an almost effective action of \(SU(2)\) on \(M\).
\end{itemize}

Moreover we \cite{MR3031643} showed the following stabilizing result:
If there is an effective action of a semi-simple compact connected Lie
group \(G\) with \(\rank G>\rank H\) on \(M\times H/T\), where \(H\) is a semi-simple compact connected Lie group with maximal torus \(T\), then the Witten genus of \(M\) vanishes.

In this note we generalize the first statement in the following way:

\begin{theorem}[Theorem \ref{sec:torus-actions-witten}]
\label{sec:introduction-2}
  Let \(M\) be a \(\text{Spin}\)-manifold such that \(p_1(M)\) is
  torsion.
  If there is an almost effective action of a torus \(T\) with \(\rank
  T > b_2(M)\) on \(M\) then the Witten genus of \(M\) vanishes.
\end{theorem}

The main new ingredient to prove this theorem is a spectral sequence argument for actions of tori \(T\) on manifolds \(M\) with \(b_2(M)<\rank T\) (see Lemma~\ref{sec:torus-acti-stab}).

If \(b_1(M)=0\), this theorem allows the following generalization, which
is also a generalization of the third statement form above.

\begin{theorem}[Theorem \ref{dirac:dis:sec:two-vanish-results}]
\label{sec:introduction-3}
  Let \(M\) be a \(\text{Spin}\)-manifold such that \(p_1(M)\) is torsion and \(b_1(M)=0\).
  Moreover, let \(M'\) be a \(2n\)-dimensional
  \(\text{Spin}^c\)-manifold, \(n>0\), with \(b_1(M')=0\) such that there are \(x_1,\dots,x_n\in H^2(M';\mathbb{Z})\) with
  \begin{enumerate}
  \item \(\sum_{i=1}^n x_i=c_1^c(M')\) modulo torsion,
  \item \(\sum_{i=1}^n x_i^2=p_1(M')\) modulo torsion,
  \item\label{dirac:dis:item:6} \(\langle \prod_{i=1}^n x_i, [M']\rangle \neq 0\).
  \end{enumerate}
If there is an almost effective action of a torus \(T\) on \(M\times M'\) such that \(\rank T > b_2(M\times M')\), then the Witten-genus of \(M\) vanishes.
Here \(c_1^c(M')\) denotes the first Chern class of the line bundle
associated to the \(\text{Spin}^c\)-structure on \(M'\).
\end{theorem}

To deduce Theorem~\ref{sec:introduction-2} from
Theorem~\ref{sec:introduction-3} in the case that \(b_1(M)=0\), let \(M'\) be \(S^2\) and \(x_1\) be
the Euler class of \(M'\). Then \(M'\) satisfies all the assumptions
from Theorem~\ref{sec:introduction-3}. Moreover there is an almost
effective action of \(T\times S^1\) on \(M\times M'\) which is induced
from the \(T\)-action on \(M\) and the \(S^1\)-action on \(M'\) given
by rotation.
Hence, the Witten genus of \(M\) vanishes, because
 \[\rank (T\times
S^1)=\rank T +1 >b_2(M) +1 =b_2(M\times M').\]

If \(H\) is a semi-simple compact connected Lie group with maximal
torus \(T'\), then the tangent bundle of \(H/T'\) splits as a sum of
complex line bundles and \(H/T'\) has positive Euler characteristic. Therefore \(H/T'\) satisfies the assumptions on
\(M'\) in the above theorem.
Hence, we get:

\begin{cor}[Corollary \ref{dirac:dis:sec:vanish-result-witt-1}]
\label{sec:introduction}
  Let \(M\) be a \(\text{Spin}\)-manifold with \(p_1(M)=0\) and
  \(b_1(M)=0\) and \(H\) a semi-simple compact connected Lie group
  with maximal torus \(T'\) and \(\dim H>0\).
  If there is an almost effective action of a torus \(T\) on \(M\times H/T'\) such that \(\rank T>\rank H + b_2(M)\),
  then the Witten-genus of \(M\) vanishes.
\end{cor}

A torus manifold is a \(2n\)-dimensional orientable manifold \(M\)  with an effective action of an \(n\)-dimensional torus \(T\) such that \(M^T\neq \emptyset\)
. 
A torus manifold \(M\) is called locally standard, if each orbit in \(M\) has an invariant neighborhood which is weakly equivariantly diffeomorphic to an open invariant subset of \(\C^n\).
Here \(\C^n\) is equipped with the action of \(T=(S^1)^n\) given by componentwise multiplication.
If this condition is satisfied the orbit space of \(M\) is naturally a manifold with corners.

A quasitoric manifold is a locally standard torus manifold whose orbit space \(M/T\) is face-preserving homeomorphic to a simple convex polytope \(P\).
Quasitoric manifolds were introduced by Davis and Januszkiewicz \cite{MR1104531}.
Torus manifolds were introduced by Masuda \cite{MR1689995} and Masuda and Hattori \cite{MR1955796} .

By combining our results with results of Dessai \cite{MR1722036}, \cite{MR1731460}  and a recent result of the author \cite{wiemeler15:_equiv} on the rigidity of certain torus manifolds, we also get the following finiteness result for simply connected torus manifolds:

\begin{theorem}[Theorem \ref{sec:appl-torus-manif}]
\label{sec:introduction-1}
  Up to homeomorphism (diffeomorphism, respectively) there are only finitely many simply connected torus manifolds \(M\) (quasitoric manifolds, respectively) with \(H^*(M;\mathbb{Z})\cong H^*(\#_{i=1}^k \pm\mathbb{C} P^n;\mathbb{Z})\) with \(k<n\).
\end{theorem}

For an application of our methods to the study of torus actions on complete intersections and homotopy complex projective spaces see \cite{dessai15:_compl_s}.

This article is structured as follows. In Section \ref{sec:preliminaries} we describe background material on vanishing results for indices of certain twisted Dirac operators on Spin\(^c\)-manifolds.
In Section \ref{sec:G/T} we prove Theorems~\ref{sec:introduction-2}
and \ref{sec:introduction-3}. Then in
Section~\ref{sec:torus-acti-stab-1} we deduce
Corollary~\ref{sec:introduction} and give some applications to
computations of the degree of symmetry of certain manifolds.
In the last Section
\ref{sec:appl-torus-manif-1} we prove Theorem~\ref{sec:introduction-1}.

\section{Preliminaries}
\label{sec:preliminaries}

In this section we recall some properties of \(2n\)-dimensional Spin$^c$-manifolds and certain twisted Dirac-operators defined on them.
For more details on this subject see \cite{0146.19001}, \cite{0247.57010}, \cite{0395.57020}, \cite{MR1722036} and \cite{MR1731460}.

A Spin$^c$-manifold \(M\) is an orientable manifold such that the second Stiefel--Whitney class \(w_2(M)\) is the reduction of an integral class \(c\in H^2(M;\mathbb{Z})\).
If this is the case then the tangent bundle of \(M\) admits a reduction of structure group to the group Spin$^c(2n)$.
We call such a reduction a Spin$^c$-structure on \(M\).
Associated to a Spin$^c$-structure there is a complex line bundle.
We denote by \(c_1^c(M)\) the first Chern-class of this line bundle.
Its mod \(2\)-reduction is \(w_2(M)\).
For each class \(c\in H^2(M;\mathbb{Z})\) with \(c\equiv w_2(M) \mod 2\) there is a Spin$^c$-structure on \(M\) with \(c_1^c(M)=c\).

Now let \(M\) be a  \(2n\)-dimensional Spin\(^c\)-manifold.
We assume that \(S^1\) acts on \(M\) and that the \(S^1\)-action lifts into the \(\text{Spin}^c\)-structure.
This is the case if and only if the \(S^1\)-action lifts into the line bundle associated to the \(\text{Spin}^c\)-structure \cite[Lemma 2.1]{MR3031643}.

Then we have an \(S^1\)-equivariant \(\text{Spin}^c\)-Dirac operator \(\partial_c\).
Its \(S^1\)-equivariant index is an element of the representation ring of \(S^1\) and is defined as
\begin{equation*}
  \ind_{S^1}(\partial_c) = \ker \partial_c - \coker \partial_c \in R(S^1).
\end{equation*}

We will discuss certain indices of twisted Dirac operators which are related to generalized elliptic genera.
Generalized elliptic genera of the type which we discuss here have first been studied by Witten \cite{MR970288}. 

Let \(V\) be a \(S^1\)-equivariant complex vector bundle over \(M\) and \(W\) an even-dimensional \(S^1\)-equivariant \(\text{Spin}\) vector bundle over \(M\).
From these bundles we construct a power series \(R\in K_{S^1}(M)[[q]]\) defined by
\begin{align*}
  R&= \bigotimes_{k=1}^\infty S_{q^k}(\tilde{TM}\otimes_\R \C)\otimes \Lambda_{-1}(V^*)\otimes \bigotimes_{k=1}^\infty \Lambda_{-q^k}(\tilde{V}\otimes_\R \C)\\& \otimes \Delta(W)\otimes\bigotimes_{k=1}^\infty \Lambda_{q^k}(\tilde{W}\otimes_\R\C).
\end{align*}
Here \(q\) is a formal variable, \(\tilde{E}\) denotes the reduced vector bundle \(E -\dim E\), \(\Delta(W)\) is the full complex spinor bundle associated to the \(\text{Spin}\)-vector bundle \(W\), and \(\Lambda_t\) (resp. \(S_t\)) denotes the exterior (resp. symmetric) power operation. The tensor products are, if not indicated otherwise, taken over the complex numbers.

We extend \(\ind_{S^1}\) to power series.
Then we can define:

\begin{definition}
  Let \(\varphi^c(M;V,W)_{S^1}\) be the \(S^1\)-equivariant index of the \(\text{Spin}^c\)-Dirac operator twisted with \(R\):
  \begin{equation*}
    \varphi^c(M;V,W)_{S^1}= \ind_{S^1}(\partial_c \otimes R)\in R(S^1)[[q]].
  \end{equation*}
We denote by \(\varphi^c(M;V,W)\) the non-equivariant version of this index:
  \begin{equation*}
    \varphi^c(M;V,W)= \ind(\partial_c \otimes R)\in \mathbb{Z}[[q]].
  \end{equation*}
\end{definition}

With the Atiyah-Singer index theorem \cite{MR0236952} we can calculate \(\varphi^c(M;V,W)\) from cohomological data: 
\begin{equation*}
  \varphi^c(M;V,W)=\langle e^{c_1^c(M)/2}\ch(R)\hat{A}(M),[M]\rangle.
\end{equation*}
Here the Chern character of \(R\) is a product
\begin{equation*}
  \ch(R)=Q_1(TM)Q_2(V)Q_3(W)
\end{equation*}
with
\begin{align*}
  Q_1(TM)&=\ch(\bigotimes_{k=1}^{\infty} S_{q^k}(\tilde{TM}\otimes_\R \C))=\prod_i\prod_{k=1}^\infty \frac{(1-q^k)^2}{(1-e^{x_i}q^k)(1-e^{-x_i}q^k)},\\
  Q_2(V)&=\ch( \Lambda_{-1}(V^*)\otimes \bigotimes_{k=1}^\infty \Lambda_{-q^k}(\tilde{V}\otimes_\R \C))\\ &= \prod_i (1-e^{-v_i})\prod_{k=1}^{\infty} \frac{(1-e^{v_i}q^k)(1-e^{-v_i}q^k)}{(1-q^k)^2},\\
  Q_3(W)&=\ch(\Delta(W)\otimes\bigotimes_{k=1}^\infty \Lambda_{q^k}(\tilde{W}\otimes_\R\C))\\ &=\prod_i(e^{w_i/2}+e^{-w_i/2})\prod_{k=1}^{\infty} \frac{(1+e^{w_i}q^k)(1+e^{-w_i}q^k)}{(1+q^k)^2},
\end{align*}
where \(\pm x_i\) (resp. \(v_i\) and \(\pm w_i\)) denote the formal roots of \(TM\) (resp. \(V\) and \(W\)).
If \(c_1^c(M)\) coincides with \(c_1(V)\), then we have
\begin{equation*}
  e^{c_1^c(M)/2}Q_2(V)= e(V)\frac{1}{\hat{A}(V)}\prod_i\prod_{k=1}^{\infty} \frac{(1-e^{v_i}q^k)(1-e^{-v_i}q^k)}{(1-q^k)^2}=e(V)Q_2'(V).
\end{equation*}

Note that if \(M\) is a \(\text{Spin}\)-manifold,
then there is a canonical \(\text{Spin}^c\)-structure on \(M\).
With respect to this \(\text{Spin}^c\)-structure the twisted index \(\varphi^c(M;0,TM)\) is equal to the elliptic genus of \(M\).
Moreover, our definition of \(\varphi^c(M;0,0)\) coincides with the index-theoretic definition of the Witten genus of \(M\).

To prove our results we need the following theorem.
It has been proven first by Liu \cite{MR1331972} for certain twisted elliptic genera of Spin-manifolds and almost complex manifolds.
Later the more general version for \(\text{Spin}^c\)-manifolds has been proven by Dessai.

\begin{theorem}[{\cite[Theorem 3.2, p. 243]{MR1722036}}]
\label{sec:twist-dirac-oper-9}
  Assume that the equivariant Pontrjagin-class \(p_1^{S^1}(V+W-TM)\) restricted to \(M^{S^1}\) is equal to \(\pi_{S^1}^*(Ix^2)\) modulo torsion, where \(\pi_{S^1}:BS^1\times M^{S^1} \rightarrow BS^1\) is the projection on the first factor, \(x\in H^2(BS^1;\mathbb{Z})\) is a generator and \(I\) is an integer.
Assume, moreover, that \(c_1^c(M)\) and \(c_1(V)\) are equal modulo torsion.
If \(I<0\), then \(\varphi^c(M;V,W)_{S^1}\) vanishes identically. 
\end{theorem}

\section{Torus actions and the Witten genus}
\label{sec:G/T}

In this section we prove Theorems \ref{sec:introduction-2} and \ref{sec:introduction-3}.
Our methods here are similar to those which were used in Section 4 of \cite{MR3031643}.

We start with a lemma.

\begin{lemma}
  \label{sec:torus-acti-stab}
  Let \(M\) be a \(T\)-manifold with \(\rank T > b_2(M)\) and \(a\in H^4_T(M;\Q)\) such that \(\iota^*a=0\in H^4(M;\Q)\).
  Then there is a non-trivial homomorphism \(\rho:S^1\rightarrow T\) such that \(\rho^*a\in \pi_{S^1}^*H^4(BS^1;\Q)\).
\end{lemma}
\begin{proof}
  From the Serre spectral sequence for the fibration \(M\rightarrow M_T\rightarrow BT\) we have the following direct sum decomposition of the \(\Q\)-vector space \(H^4_T(M;\Q)\),
  \begin{equation*}
    H^4_T(M;\Q)\cong E^{0,4}_\infty \oplus E^{2,2}_\infty \oplus E^{4,0}_\infty.
  \end{equation*}
Moreover, we have
\begin{align*}
  E^{0,4}_\infty &\subset H^{4}(M;\Q)& E^{2,2}_\infty &\subset E^{2,2}_2/d_2(E^{0,3}_2)& E^{4,0}_\infty&=\pi_{S^1}^*H^4(BT;\mathbb{Q}).
\end{align*}
Let \(a_{0,4}\), \(a_{2,2}\), \(a_{4,0}\) be the components of \(a\) according to this decomposition.
Then \(a_{0,4}=0\) by assumption.
Moreover, there is an \(\tilde{a}_{2,2}\in E_2^{2,2}\) such that \(a_{2,2}=[\tilde{a}_{2,2}]\).

Now it is sufficient to find a non-trivial homomorphism \(\rho:S^1\rightarrow T\) such that \(\rho^*\tilde{a}_{2,2}=0\).
We have isomorphisms
\begin{align*}
  E_2^{2,2}&\cong H^2(BT;\mathbb{Q})\otimes H^2(M;\mathbb{Q})\\
  &\cong \left(H^2(BT;\Q)\right)^{b_2(M)}
\end{align*}
Since \(\rank T > b_2(M)\) we can find a non-trivial homomorphism \(\phi: H^2(BT;\Q)\rightarrow H^2(BS^1;\Q)=\Q\) such that all components of \(\tilde{a}_{2,2}\) according to the above decomposition of \(E_2^{2,2}\) are mapped to zero by \(\phi\).
After scaling, we may assume that \(\phi\) is induced by a surjective homomorphism \(H^2(BT;\mathbb{Z})\rightarrow H^2(BS^1;\mathbb{Z})\).
By dualizing we get a homomorphism \(\hat{\phi}: H_2(BS^1;\mathbb{Z})\rightarrow H_2(BT;\mathbb{Z})\).
Since for any torus \(H_2(BT;\mathbb{Z})\) is naturally isomorphic to the integer lattice in the Lie algebra \(LT\) of \(T\), \(\hat{\phi}\) defines the desired homomorphism.
\end{proof}

By combining this lemma with the above result of Liu and Dessai
(Theorem~\ref{sec:twist-dirac-oper-9}) we get the following theorem.

\begin{theorem}
\label{sec:torus-actions-witten}
  Let \(M\) be a \(\text{Spin}\)-manifold such that \(p_1(M)\) is
  torsion.
  If there is an almost effective action of a torus \(T\) with \(\rank
  T > b_2(M)\) on \(M\) then the Witten genus \(\varphi^c(M;0,0)\) of \(M\) vanishes.
\end{theorem}
\begin{proof}
  First note that, by replacing the \(T\)-action by the action of a
  double covering group of \(T\), we may assume that the \(T\)-action
  lifts into the \(\text{Spin}\)-structure of \(M\).

Therefore, by Theorem \ref{sec:twist-dirac-oper-9}, it is sufficient to show that there is a homomorphism \(\rho:S^1\hookrightarrow T\) such that \(\rho^*p_1^T(- TM)=a x^2\), where \(x\in H^2(BS^1;\mathbb{Z})\) is a generator and \(a\in \mathbb{Z}\), \(a<0\).
By Lemma~\ref{sec:torus-acti-stab}, there is a homomorphism \(\rho:S^1\rightarrow T\) such that
\begin{equation*}
 p_1^{S^1}(- TM)=\rho^*p_1^T(- TM)=a x^2 
\end{equation*}
  with \(a \in\mathbb{Z}\).

Moreover, we have
\begin{equation*}
  a x^2=p_1^{S^1}(- TM)|_y = -\sum v_i^2,
\end{equation*}
where \(y\in M^T\) is a \(T\) fixed point and the \(v_i\in
H^2(BS^1;\mathbb{Z})\) are the weights of the \(S^1\)-representation
\(T_yM\).
We may assume that such a fixed point \(y\) exists because otherwise
the Witten genus of \(M\) vanishes by an application of the Lefschetz
fixed point formula.

Not all of the \(v_i\) vanish because the \(T\)-action on \(M\) is almost effective, which implies that the \(S^1\)-action on \(M\) is non-trivial.
Therefore the theorem is proved.
\end{proof}

We can also deduce the following partial generalization of the above
result. Its proof is similar to the proof of Theorems 4.1 and 4.4 in \cite{MR3031643}.
These theorems are concerned with actions of semi-simple and simple compact connected Lie-groups.
Whereas the theorem which we present here deals with torus actions.

\begin{theorem}
\label{dirac:dis:sec:two-vanish-results}
  Let \(M\) be a \(\text{Spin}\)-manifold such that \(p_1(M)\) is torsion and \(b_1(M)=0\).
  Moreover, let \(M'\) be a \(2n\)-dimensional
  \(\text{Spin}^c\)-manifold, \(n>0\), with \(b_1(M')=0\) such that there are \(x_1,\dots,x_n\in H^2(M';\mathbb{Z})\) with
  \begin{enumerate}
  \item \(\sum_{i=1}^n x_i=c_1^c(M')\) modulo torsion,
  \item \(\sum_{i=1}^n x_i^2=p_1(M')\) modulo torsion,
  \item\label{dirac:dis:item:6} \(\langle \prod_{i=1}^n x_i, [M']\rangle \neq 0\).
  \end{enumerate}
If there is an almost effective action of a torus \(T\) on \(M\times M'\) such that \(\rank T > b_2(M\times M')\), then the Witten-genus \(\varphi^c(M;0,0)\) of \(M\) vanishes.
\end{theorem}
\begin{proof}
Let \(L_i\), \(i=1,\dots,n\), be the line bundle over \(M'\) with \(c_1(L_i)=x_i\).
Because \(b_1(M\times M')=0\), the natural map \(\iota^*:H^2_T(M\times M';\mathbb{Z})\rightarrow H^2(M\times M';\mathbb{Z})\) is surjective.

Therefore by Corollary 1.2 of \cite[p. 13]{MR0461538} the \(T\)-action on \(M\times M'\) lifts into \(p'^*(L_i)\), \(i=1,\dots,n\).
Here \(p': M\times M'\rightarrow M'\) is the projection.
We can choose these lifts in such a way that the torus action on the fibers of  \(p'^*(L_i)\), \(i=1,\dots,n\), over a fixed point \(y\in (M\times M')^T\) are trivial.
Moreover, by the above cited corollary and Lemma 2.1 of \cite{MR3031643}, the action of every \(S^1\subset T\) lifts
 into the \(\text{Spin}^c\)-structure on \(M\times M'\) induced by the \(\text{Spin}\)-structure on \(M\) and the \(\text{Spin}^c\)-structure on \(M'\).

By Lemma 3.1 of \cite{MR3031643}, we have
\begin{equation*}
  \varphi^c(M\times M';\bigoplus_{i=1}^n p'^*L_i,0)=\varphi^c(M;0,0)\varphi^c(M';\bigoplus_{i=1}^n L_i,0).
\end{equation*}
By condition (\ref{dirac:dis:item:6}), we have
\begin{align*}
  \varphi^c(M';\bigoplus_{i=1}^n L_i,0)&=\langle Q_1(TM')\prod_{i=1}^n x_i Q_2'(\bigoplus_{i=1}^nL_i) \hat{A}(M'),[M']\rangle\\
&= \langle \prod_{i=1}^n x_i,[M']\rangle\neq 0.
\end{align*}
Hence, \(\varphi^c(M;0,0)\) vanishes if and only if \(\varphi^c(M\times M';\bigoplus_{i=1}^n p'^*L_i,0)\) vanishes.

By Theorem \ref{sec:twist-dirac-oper-9}, it is sufficient to show that there is a homomorphism \(\rho:S^1\hookrightarrow T\) such that \(\rho^*p_1^T(\bigoplus_{i=1}^n p'^*L_i - T(M\times M'))=a x^2\), where \(x\in H^2(BS^1;\mathbb{Z})\) is a generator and \(a\in \mathbb{Z}\), \(a<0\).
By Lemma~\ref{sec:torus-acti-stab}, there is a homomorphism \(\rho:S^1\rightarrow T\) such that
\begin{equation*}
 p_1^{S^1}(\bigoplus_{i=1}^n p'^*L_i - T(M\times M'))=\rho^*p_1^T(\bigoplus_{i=1}^n p'^*L_i - T(M\times M'))=a x^2 
\end{equation*}
  with \(a \in\mathbb{Z}\).

Moreover, we have
\begin{equation*}
  a x^2=p_1^{S^1}(\bigoplus_{i=1}^n p'^*L_i - T(M\times M'))|_y = \sum_{i=1}^n a_i^2 -\sum v_i^2,
\end{equation*}
where the \(a_i\in H^2(BS^1;\mathbb{Z})\), \(i=1,\dots,n\), are the weights of the \(S^1\)-representations \(p'^*L_i|_y\) and the \(v_i\in H^2(BS^1;\mathbb{Z})\) are the weights of the \(S^1\)-representation \(T_y(M\times M')\).
By our choice of the lifted actions the \(a_i\) vanish.
Not all of the \(v_i\) vanish because the \(T\)-action on \(M\)  is effective, which implies that the \(S^1\)-action on \(M\) is non-trivial.
Therefore the theorem is proved.
\end{proof}

Examples of manifolds \(M'\) to which the above theorem applies are manifolds whose tangent bundles split as Whitney sums of complex line bundles and which have non-zero Euler-characteristic.
In particular, if \(H\) is a semi-simple compact connected Lie-group
with maximal torus \(T'\) and \(\dim H > 0\), then \(M'=H/T'\) satisfies these assumptions.
We deal with this case in the following section.

\section{Torus actions and stabilizing with $G/T$}
\label{sec:torus-acti-stab-1}

In this section we deal with applications of Theorem
\ref{dirac:dis:sec:two-vanish-results} to the particular case where
\(M'\) is a homogeneous space \(H/T'\) with \(H\) a semi-simple
compact connected Lie group and \(T'\) a maximal torus of \(H\) and
\(\dim H>0\).

It has already been noted that the tangent bundle of \(H/T'\) splits
as a sum of complex line bundles.
Therefore \(H/T'\) satisfies all the assumptions on \(M'\) from
Theorem~\ref{dirac:dis:sec:two-vanish-results}.
Hence we immediately get the following corollary.

\begin{cor}
\label{dirac:dis:sec:vanish-result-witt-1}
  Let \(M\) be a \(\text{Spin}\)-manifold with \(p_1(M)=0\) and
  \(b_1(M)=0\) and \(H\) a semi-simple compact connected Lie group
  with maximal torus \(T'\) and \(\dim H > 0\).
  If there is an almost effective action of a torus \(T\) on \(M\times H/T'\) such that \(\rank T>\rank H + b_2(M)\),
  then the Witten-genus of \(M\) vanishes.
\end{cor}

The degree of symmetry \(N(M)\) of a manifold \(M\) is the maximum of the dimensions of compact connected Lie groups \(G\) which act smoothly and almost effectively on \(M\).
By combining the above corollary with Corollary~4.2 of \cite{MR3031643} we get the following bounds for the degree of symmetry of the manifolds \(M\times H/T'\).
To state our result we have to introduce some notation.
For \(l\geq 1\) let
\begin{equation*}
  \alpha_l=\max\left\{\frac{\dim G}{\rank G}; \; G \text{ a simple compact Lie-group with } \rank G \leq l\right\}.
\end{equation*}
The values of the \(\alpha_l\)'s are listed in Table \ref{tab:erste}.

\begin{table}
  \centering
  \begin{tabular}{|c|c|c|}
    \(l\)&\(\alpha_l\)& \(G_l\)\\\hline\hline
    \(1\)& \(3\)&\(\text{Spin}(3)\)\\\hline
    \(2\)& \(7\)&\(G_2\)\\\hline
    \(3\)& \(7\)&\(\text{Spin}(7), Sp(3)\)\\\hline
    \(4\)& \(13\)&\(F_4\)\\\hline
    \(5\)& \(13\)& none\\\hline
    \(6\)& \(13\)& \(E_6, \text{Spin}(13), Sp(6)\)\\\hline
    \(7\)& \(19\)& \(E_7\)\\\hline
    \(8\)& \(31\)& \(E_8\)\\\hline
    \(9\leq l\leq 14\)& \(31\)& none\\\hline
    \(l\geq 15\)&\(2l+1\)&\(\text{Spin}(2l+1), Sp(l)\)\\
  \end{tabular}
\caption{The values of \(\alpha_l\) and the simply connected compact simple Lie-groups \(G_l\) of rank \(l\) with \(\dim G_l= \alpha_l\cdot l\).}
\label{tab:erste}
\end{table}

\begin{cor}
\label{dirac:dis:sec:two-vanish-results-2}
  Let \(M\) be a \(\text{Spin}\)-manifold with \(p_1(M)=0\) and \(b_1(M)=0\), such that the Witten-genus of \(M\) does not vanish and \(H_1,\dots, H_k\) simple compact connected Lie groups with maximal tori \(T_1,\dots,T_k\).
  Then we have
  \begin{equation*}
    \sum_{i=1}^k \dim H_i \leq N(M\times \prod_{i=1}^k H_i/T_i) \leq \alpha_l\sum_{i=1}^k\rank H_i+b_2(M),
  \end{equation*}
where \(l=\max\{\rank H_i;\; i=1,\dots,k\}\) and \(\alpha_l\) is defined as above.
\end{cor}
\begin{proof}
  Let \(G\) be a compact connected Lie group which acts almost effectively on \(M\times \prod_{i=1}^k H_i/T_i\).
  We may assume that \(G=G_{ss} \times Z\) with a semi-simple Lie group \(G_{ss}\) and a torus \(Z\).

By Corollary \ref{dirac:dis:sec:vanish-result-witt-1}, \(\rank G\) is bounded from above by \(\sum_{i=1}^k \rank H_i + b_2(M)\).
By Corollary 4.2 of \cite{MR3031643}, \(\rank G_{ss}\) is bounded from above by \(\sum_{i=1}^k \rank H_i\).
Moreover, by the proof of Corollary~4.6 of \cite{MR3031643} the dimension of \(G_{ss}\) is bounded from above by \(\alpha_l \rank G_{ss}\).
Since \(\alpha_l>1\), it follows that
\begin{align*}
  \dim G&=\dim G_{ss}+\dim Z = \dim G_{ss} +\rank G-\rank G_{ss}\\
  &\leq (\alpha_l-1)\rank G_{ss} + \sum_{i=1}^k \rank H_i + b_2(M)\\
  &\leq \alpha_l\sum_{i=1}^k\rank H_i+b_2(M). 
\end{align*}
This proves the second inequality. The first inequality is trivial.
\end{proof}

Note that if in the situation of Corollary~\ref{dirac:dis:sec:two-vanish-results-2} the groups \(H_i\) are all equal to one of the  groups listed in table~\ref{tab:erste} and are all isomorphic and \(b_2(M)=0\), then the left and right hand sides of the inequality in Corollary~\ref{dirac:dis:sec:two-vanish-results-2} are equal. Therefore in this case the degree of symmetry of \(M\times \prod_{i=1}^k H_i/T\) is equal to \(\dim \prod_{i=1}^kH_i\).
This leads to the following corollary.

\begin{cor}
  Let \(G\) be \(\text{Spin}(2l+1)\), \(Sp(l)\) with \(l\geq 15\) or an exceptional simple compact connected Lie-group with maximal torus \(T\).
 Moreover, let \(M\) be a two-connected manifold with \(p_1(M)=0\) and non-zero Witten genus. 
Then we have
\begin{equation*}
  N(M\times \prod_{i=1}^k G/T)= k \dim G.
\end{equation*}
\end{cor}

\section{An application to torus manifolds}
\label{sec:appl-torus-manif-1}

In this section we prove the following theorem.

\begin{theorem}
\label{sec:appl-torus-manif}
  Up to homeomorphism (diffeomorphism, respectively) there are only finitely many simply connected torus manifolds \(M\) (quasitoric manifolds, respectively) with \(H^*(M;\mathbb{Z})\cong H^*(\#_{i=1}^k \pm\mathbb{C} P^n;\mathbb{Z})\) with \(k<n\).
\end{theorem}

Note that if \(\dim M< 6\) then this theorem follows directly from the classification of simply connected torus manifolds of dimension four given by Orlik and Raymond \cite{MR0268911} and the fact that the sphere is the only two-dimensional torus manifold.

In higher dimensions the proof of the theorem is subdivided into two lemmas.
The first one is:

\begin{lemma}
\label{sec:an-aplication-torus-1}
  Let \(M\) be a simply connected torus manifold (a quasitoric manifold, respectively), with \(H^*(M;\mathbb{Z})\cong H^*(\#_{i=1}^k \pm\mathbb{C} P^n;\mathbb{Z})\), \(k\in\mathbb{N}\), \(n\geq 3\).
  Then up to finite ambiguity the homeomorphism type (diffeomorphism type, respectively) is determined by the first Pontrjagin class of \(M\).
\end{lemma}
\begin{proof}
  By Theorem 1.1 of \cite{wiemeler15:_equiv}, Theorem 2.2 of \cite{MR2885534}   and Theorem 3.6 of \cite{wiemeler15:_torus}, it is sufficient to prove that the Poincar\'e duals of the characteristic submanifolds of \(M\) are determined up to finite ambiguity by \(p_1(M)\).
  The characteristic submanifolds of \(M\) are codimension two submanifolds which are fixed by circle subgroups of the torus which acts on \(M\).
  Let \(u_1,\dots,u_m\in H^2(\#_{i=1}^k \pm\mathbb{C} P^n;\mathbb{Z})\) be their Poincar\'e duals.
  Moreover, we have
  \begin{equation*}
    H^*:=H^*(\#_{i=1}^k \pm\mathbb{C} P^n;\mathbb{Z})=\mathbb{Z}[v_1,\dots,v_k]/(v_iv_j, v_i^n\pm v_j^n;\; 1\leq i<j\leq k)
  \end{equation*}
  with \(\deg v_i=2\) for \(i=1,\dots,k\).

  Therefore there are \(\alpha_{ij}\in\mathbb{Z}\) such that \(u_i=\sum_{j=1}^k \alpha_{ij} v_j\).
  
Since \(M\) is equivariantly formal, it follows from localization in equivariant cohomology that
\begin{equation*}
  p_1(M)=\sum_{i=1}^m u_i^2=\sum_{j=1}^k\left(\sum_{i=1}^m \alpha_{ij}^2\right) v_j^2.
\end{equation*}
Because the \(v_j^2\) form a basis of \(H^4\) it follows that for fixed \(p_1(M)\) there are only finitely many possibilities for the \(\alpha_{ij}\). Therefore the \(u_i\) are contained  in a finite set which only depends on \(p_1(M)\).
  
This proves the lemma.  
\end{proof}

The second lemma is as follows:

\begin{lemma}
\label{sec:an-aplication-torus}
  Let \(M\) be a torus manifold, with \(H^*(M;\mathbb{Z})\cong H^*(\#_{i=1}^k \pm\mathbb{C} P^n;\mathbb{Z})\), \(k<n\), \(n\geq 3\).

  Then with the notation from the proof of the previous lemma we have
  \begin{equation*}
    p_1(M)=\sum_{i=1}^k\beta_iv_i^2
  \end{equation*}
with \(0< \beta_i\leq n+1\).
\end{lemma}
\begin{proof}
  The inequality \(0<\beta_i\) follows from the formula for \(p_1(M)\) given in the proof of the previous lemma.
  Therefore we only have to show that for all \(i\), \(\beta_i\leq n+1\).

  Assume the contrary, i.e. \(\beta_{i_0}>n+1\) for some \(i_0\in \{1,\dots,k\}\).
  Since the natural map \(H^2(M;\mathbb{Z})\rightarrow H^2(M;\mathbb{Z}_2)\) is surjective, \(M\) is a Spin\(^c\)-manifold.
  Let \(\alpha_{i}\in \{0,1\}\), \(i=1,\dots,k\) such that \(w_2(M)\equiv\sum_{i=1}^k \alpha_iv_i\mod 2\).

  Then there are two cases \(\alpha_{i_0}\equiv n+1\mod 2\) and \(\alpha_{i_0}\equiv n\mod 2\).
  
  We first deal with the first case.
  Choose a Spin\(^c\)-structure on \(M\) such that \(c_1^c(M)= (n+1)v_{i_0}+\sum_{i\neq i_0} \alpha_i v_i\).
  Because \(b_1(M)=0\) every \(S^1\)-action on \(M\) lifts into this Spin$^c$-structure and into all line-bundles over \(M\).
  We can choose these lifts in such a way that the actions on the fiber of a line bundle over a given fixed point \(y \in M^{S^1}\) is trivial. 
  By the relation \(w_2(M)^2\equiv p_1(M)\mod 2\), we know that
    \(\beta_{i}\equiv \alpha_i^2\mod 2\).
    Therefore we have \(\beta_{i_0}\geq n+3\).
    Now for \(x\in H^2(M;\mathbb{Z})\) let \(L(x)\) be the line bundle over \(M\) with first Chern class \(x\).
    Moreover, let
    \begin{equation*}
      V=L(2v_{i_0})\oplus L(v_{i_0}+\sum_{i\neq i_0}\alpha_i v_i) \oplus (n-2)L(v_{i_0})
    \end{equation*}
    and
    \begin{equation*}
      W=\bigoplus_{i\neq i_0}(\beta_i-\alpha_i) L(v_i) \oplus (\beta_{i_0}-n-3) L(v_{i_0}).
    \end{equation*}
    Then we have \(c_1(V)=c_1^c(M)\), \(p_1(V\oplus W\ominus TM)=0\) and \(W\) is a spin bundle.

    Therefore as in the proof of Theorem \ref{dirac:dis:sec:two-vanish-results} it follows from Theorem~\ref{sec:twist-dirac-oper-9} and Lemma~\ref{sec:torus-acti-stab}, that \(\varphi^c(M;V,W)=0\) if \(k<n\).
    This gives a contradiction since a direct computation shows that
    \begin{equation*}
      \varphi^c(M;V,W)=\langle e(V),[M]\rangle=\pm 2\neq 0.
    \end{equation*}

    The case where \(\alpha_{i_0}\equiv n\mod 2\) is similar.
    In this case one has to choose a Spin$^c$-structure on \(M\) such that \(c_1^c(M)=nv_{i_0}+\sum_{i\neq i_0} \alpha_i v_i\).
    Moreover one has to consider the bundles
    \begin{equation*}
      V=L(v_{i_0}+\sum_{i\neq i_0}\alpha_i v_i) \oplus (n-1)L(v_{i_0})
    \end{equation*}
    and
    \begin{equation*}
      W=\bigoplus_{i\neq i_0}(\beta_i-\alpha_i) L(v_i) \oplus (\beta_{i_0}-n) L(v_{i_0}).
    \end{equation*}
    The details are left to the reader.
\end{proof}

Now Theorem~\ref{sec:appl-torus-manif} follows directly from Lemmas~\ref{sec:an-aplication-torus-1} and \ref{sec:an-aplication-torus}.

\bibliography{torus_witten}{}
\bibliographystyle{alpha}
\end{document}